\documentclass[12pt,reqno]{amsart}

\usepackage{amsmath,amssymb,amscd,graphicx, amsthm,dsfont,
  enumerate, hyperref} 
\calclayout
\usepackage[pagewise]{lineno}

\usepackage{mathrsfs}
\usepackage{mathtools}
\usepackage{bigints}
\usepackage[all]{xy}
\usepackage{amsmath}
\usepackage{color,xcolor}
\definecolor{darkred}{rgb}{1,0,0} 
\definecolor{darkgreen}{rgb}{0,0.8,0}
\definecolor{darkblue}{rgb}{0,0,1}
\hypersetup{colorlinks,
linkcolor=darkblue,
filecolor=darkgreen,
urlcolor=darkred,
citecolor=darkgreen}
\numberwithin{equation}{section}
\theoremstyle{plain}
\newtheorem{theorem}{Theorem}
\numberwithin{theorem}{section}
\newtheorem{proposition}[theorem]{Proposition}
\newtheorem{lemma}[theorem]{Lemma}

\newtheorem{corollary}[theorem]{Corollary}
\theoremstyle{definition}

\title{A Sobolev Space Property of Logarithm of Lipschitz functions}

\author{Yifei Pan}
\address{Department of Mathematical Sciences, Purdue University Fort Wayne, 2101 East Coliseum Boulevard,
Fort Wayne, IN 46805, USA}
\email{pan@pfw.edu}

\begin{document}
 
 \begin{abstract}
For a Lipschitz function $f$ on an open set in $\mathbb{R}^n$, we consider the $L^{n}$ integrability of the quotient $\frac{|\nabla f|}{|f|}$ over its natural domain of definition. 
\end{abstract}

\maketitle

\section{Introduction and Results}

In this note, we prove the following seemingly simple result concerning the logarithm of Lipschitz functions.

\begin{theorem}\label{t1}
Let $\Omega$ be an open set in $\mathbb{R}^n$, and let $f$ be a Lipschitz function on $\overline\Omega$. Then the function $\log |f(x)-f(a)|$ does not belong to  the Sobolev space $W^{1,n}(\Omega)$ for every $a\in\overline\Omega$. Furthermore, if the function $\log |f(x)|$ belongs to the Sobolev space $W^{1,n}(\Omega)$, then $f$ never vanishes in the closure of $\Omega$. 

\end{theorem}
A natural way to view this result is to consider a (non-linear ) map $T_a(f)=\log |f(x)-f(a)|$ defined on all Lipschitz functions $Lip(\Omega)$ for each $a\in \overline{\Omega}$. Then
 $$T_a(Lip(\Omega))\cap  W^{1,n}(\Omega)=\emptyset .$$
The proof of this result is based on a general blow-up phenomenon, as we shall prove below, for Lipschitz functions. This phenomenon seems to reveal a competition between the sets of critical points and zeros at least for continuously differentiable functions.
\begin{theorem}\label{t1}
Let $\Omega$ be an open set in $\mathbb{R}^n$, and let $f$ be a non-constant locally Lipschitz  function on $\Omega$. If the zero set $\{x\in\Omega: f(x)=0\}$ of $f$ is not empty, then 
\begin{equation}\label{po}
\int_{\Omega\setminus f^{-1}(0)}\bigg|\frac{\nabla f(x)}{f(x)}\bigg|^{n}\, dx=\infty,
\end{equation}
or equivalently 
\begin{equation}\label{po}
\int_{\Omega\setminus f^{-1}(0)}|\nabla \log|f(x)||^n\, dx=\infty.
\end{equation}
\end{theorem}

 Throughout, we denote the gradient $\nabla f$ of a Lipschitz function $f$, which exists a.e. by Rademacher's theorem, and its norm $|\nabla f|^2=f^2_{x_1}+...+f^2_{x_n}$. Also, it is a well-known fact that a (locally integrable) function is locally Lipschitz if and only if its distributional gradient $\nabla f$ is locally $L^{\infty}$.
On the other hand, any Lipschitz function on a set of a metric measure space can be extended to be a Lipschitz function over the whole space with the same Lipschitz constant; in particularly, for our purpose in this paper, we can always use the extension of a Lipschitz function $f$ in $\Omega$ to one in $\mathbb{R}^n$ with the same Lipschitz constant [4].

An immediate corollary of Theorem 1.2 is a uniqueness theorem of differential inequality of the gradient.

\begin{corollary}
Let $f$ be a locally Lipschitz function in an open set $\Omega$ satisfying
\begin{equation}
|\nabla f(x)|\leq V(x) |f(x)|,\,\, x\in\Omega, a.e.,
\end{equation}
 where $V\in L_{loc}^n(\Omega)$. If there is a point $x_0\in\Omega$ such that $f(x_0)=0$, then $f\equiv 0$ in $\Omega$.
\end{corollary}

\begin{corollary}Let $\Omega$ be an open set in $\mathbb{R}^n$, and let $f$ be a non-constant locally Lipschitz  function on $\Omega$. If the zero set $\{x\in\Omega: f(x)=0\}$ of $f$ is not empty and  
\begin{equation}\label{po}
\int_{\Omega\setminus f^{-1}(0)}|\nabla \log|f(x)||^p\, dx<\infty.
\end{equation}
then, $p<n$
\end{corollary}

\begin{corollary}Let $\Omega$ be an open set in $\mathbb{R}^n$, and let $f\in W^{1,n}(\Omega)$. Then the exponential $e^f$ of $f$ cannot be Lipschitz unless $f$ is Lipschitz and $f\not=0$ in $\overline\Omega$.
\end{corollary}

Theorem 1.2 is somehow cosmetically related to the following well-known result of J. Bourgain, H. Brezis and P. Mironescu [3].
\begin{theorem}\label{t2}[3]
Let $\Omega$ be a connected open set in $\mathbb{R}^n$, and let $f:\Omega\to \mathbb{R}$ be a non-constant measurable function. Then
$$\int_\Omega\int_\Omega \frac{|f(x)-f(y)|}{|x-y|^{n+1}}dxdy=\infty.$$
\end{theorem}
We make some remarks on the results above, which seem to be easy to state, but not exactly obvious to prove. First the integrability exponent $n$, the dimension of the space is the best possible for the results to hold if we take, for example, the function $f(x)=|x|^2$. Secondly, for positive functions, the integral can be arbitrarily small.
 In \cite{CGW93} A. Chang, M. Gursky, and T. Wolff proved the following calculus result, which is crucial for them to construct a counterexample of a geometric problem.
\begin{proposition}\cite[page 144]{CGW93}\label{PropCGW}
Suppose $n\geq 3$. Then for any $R>0, A, B>0, \epsilon>0$, there are $\delta>0$ and a smooth radial function $\psi\colon \mathbb{R}^n\to (0,+\infty)$ such that $\psi(x)=A$ when $|x|\geq R, \psi(x)=B$ when $|x|\leq \delta$, $\min(A,B)\leq \psi\leq \max(A,B)$, and 
\begin{equation}\label{integralsInCGW}
\int_{|x|\leq R}\Big{|}\psi^{-1}\frac{d\psi}{dx^i}\Big{|}^{n}<\epsilon \quad\quad\quad 1\leq i\leq n.
\end{equation}
\end{proposition}
This example is, equivalently, saying that there are positive functions for which the integral
 $$\int_{|x|<R}|\nabla \log|\psi(x)||^n\, dx$$ 
 can be made arbitrarily small. On the other hand,
Theorem \ref{t1} has the following interesting interpretation as an extension result: for every Lipschitz $f$  and every $a\in \Omega$, the function $\frac{|\nabla f(x)|}{|f(x)-f(a)|}$ cannot be extended to a function in $L^n(\Omega)$. We will actually prove that the integral in \eqref{po} diverges around any point of the boundary of $f^{-1}(f(a))$. Hence $\frac{|\nabla f(x)|}{|f(x)-f(a)|}$ cannot be extended to an $L^n$ function beyond its natural domain of definition $\Omega\setminus f^{-1}(f(a))$. 

As a curious consequence, given any closed set $A\subset\mathbb{R}^n$, it is possible to construct a continuous function on $\mathbb{R}^n\setminus A$ that cannot be extended locally to a $L^n$ function past any point of the boundary of $A$. Indeed, one just needs to consider the quotient $\frac{|\nabla f|}{|f|}$, where $f$ is a smooth function in $\mathbb{R}^n$ whose zeroset is $A$. Note that such an $f$ exists by the Whitney extension theorem [2].


\section{Proof of Theorem 1.2}
We first begin with a uniqueness theorem for Lipschitz functions in one real variable over intervals which is tailored for applications in polar coordinates in order to prove the results in this paper. Recall that a Lipschitz function in $[0,1]$ is differentiable almost everywhere and satisfies $f(b)-f(a)=\int_a^b f'(x)dx$ for any $a, b\in [0,1]$.
\begin{lemma}
Let $\varphi$ be a Lipschitz function over $[0,1]$ with $\varphi(0)=0$. Assume that there exists $p\geq 1$ and a non-negative function $\lambda\in L^p(0,1)$ such that\begin{equation}\label{h}
|\varphi'(x)|\leq \lambda(x)\,|\varphi(x)|\, x^{\frac{1-p}{p}}\quad\quad \forall x\in(0,1), a.e.
\end{equation}
Then $\varphi\equiv 0$ in $[0,1]$.
\end{lemma}
\begin{proof} Let $\delta=\sup\{d\in[0,1]\,\vert\, \varphi\equiv 0 \text{ in }[0,d]\}$. By the fundamental theorem for Lipschitz functions and H\"older's inequality, we have
\begin{equation}
|\varphi(x)|\leq\int_{\delta}^{x}|\varphi'(t)|\, dt\leq\bigg(\int_{\delta}^x|\varphi'(t)|^p\, dt\bigg)^{\frac{1}{p}}\bigg(\int_{\delta}^x 1^q\, dt\bigg)^{\frac{1}{q}},\quad\quad \frac{1}{p}+\frac{1}{q}=1.
\end{equation}
Hence
\begin{equation}\label{1}
|\varphi(x)|^p\leq x^{\frac{p}{q}}\int_{\delta}^x|\varphi'(t)|^p\, dt=x^{p-1}\int_{\delta}^x|\varphi'(t)|^p\, dt.
\end{equation}
We multiply both sides of \eqref{1} by the function $\lambda^p(x)$. Note that we can assume without loss of generality that $\lambda$ is non-vanishing. Indeed, if that is not the case, we can just replace $\lambda$ with $1+\lambda$. We obtain
\begin{equation}\label{2}
\lambda^p(x)\,|\varphi(x)|^p\leq \,x^{p-1}\,\lambda^p(x)\int_{\delta}^x|\varphi'(t)|^p\, dt.
\end{equation}
Let $s\in (\delta,1)$. Integrating in the variable $x$ on both sides of \eqref{2} gives
\begin{equation}\label{3}
\int_{\delta}^s\lambda^p(x)\,|\varphi(x)|^p\, x^{1-p}\, dx\leq \int_{\delta}^s\bigg{(}\lambda^p(x)\int_{\delta}^x|\varphi'(t)|^p\, dt\bigg{)}\,dx.
\end{equation}
Since $x\leq s$, then \eqref{3} implies
\begin{equation}\label{6}
\int_{\delta}^s\lambda^p(x)\,|\varphi(x)|^p\, x^{1-p}\, dx\leq \bigg(\int_{\delta}^s \lambda^p(x)\, dx\bigg)\bigg(\int_{\delta}^s|\varphi'(x)|^p\, dx\bigg). 
\end{equation}
By \eqref{h}, we have that 
\begin{equation}\label{7}
\int_{\delta}^s\lambda^p(x)\,|\varphi(x)|^p\, x^{1-p}\, dx\leq  \bigg(\int_{\delta}^s \lambda^p(x)\, dx\bigg)\,\bigg(\int_{\delta}^s\lambda^p(x)\,|\varphi(x)|^p\, x^{1-p}\, dx\bigg).
\end{equation}
Since $\varphi$ is Lipschitz and $\varphi(0)=0$, we have that $|\varphi(x)|\leq C|x|$ for some $C$ for $x\in [0,1]$, and therefore the function $|\varphi(x)|^p\, x^{1-p}$ is bounded. Since $\lambda\in L^p(0,1)$, we conclude that
\begin{equation}
\int_{\delta}^s\lambda^p(x)\,|\varphi(x)|^p\, x^{1-p}dx<\infty \quad\quad \forall s\in(\delta,1).
\end{equation}
 and we can find a sequence of points $s_j\in (\delta,1), s_j\to \delta$, such that
\begin{equation}\label{8}
\int_{\delta}^{s_j}\lambda^p(x)\,|\varphi(x)|^p\, x^{1-p}\, dx\neq 0 \quad \forall j.
\end{equation}  
Equation \eqref{7} then yields 
\begin{equation}\label{final}
1\leq\int_{\delta}^{s_j}\lambda^p(x)\,dx \quad \forall j.
\end{equation}
Since $\lambda\in L^p(0,1)$, then letting $s_j\to \delta$ in \eqref{final} leads to a contradiction. 
\end{proof}
Rademacher's theorem says a Lipschitz function is differentiable almost everywhere. The following simple, but useful lemma tells where the square of a Lipschitz function could be differentiable.
\begin{lemma}
 Let $f$ be a locally Lipschitz function in a domain. Then the square function $g=f^2$ is differentiable wherever $f$ vanishes and in fact $\nabla g(x)=0$ there. 
\end{lemma}
\begin{proof} Let $x_0$ be such that $f(x_0)=0$. We claim that $g=f^2$ is differentiable at $x_0$ and $\nabla g(x_0)=0$. Indeed, by the definiton of differentiability,
$$\lim_{x\to x_0}\frac{g(x)-g(x_0)-0\cdot(x-x_0)}{|x-x_0|}=\lim_{x\to x_0}\frac{f^2(x)}{|x-x_0|}=0,$$
where we have used $f(x)=f(x)-f(x_0)=O(|x-x_0|)$ because of  $f$ being locally Lipschitz at $x_0$.
\end{proof}
\begin{lemma} Let $A$ be a measurable set of measure zero in the unit ball in $\mathbb{R}^n$. Then for almost all $\omega\in S^{n-1}$,the unit sphere, the set of intersection of $A$ with the ray $\{r\omega: 0\leq r\leq 1\}$ is of measure zero in the line measure.
\end{lemma}
\begin{proof}
Let $\chi_A$ be the characteristic function of the set $A$. We have $0=|A|=\int_{|x|<1} \chi_A(x)dx=\int_{S^{n-1}}\int_0^1\chi_A(r\omega)r^{n-1}drd\omega$. By Fubini's theorem, we conculde that for almost all $\omega\in S^{n-1}$, $\int_0^1\chi_A(r\omega)r^{n-1}dr=0$, which is the desired result: $|A\cap\{r\omega\}|=0$.
\end{proof}
\begin{theorem}
Let $\Omega\subset\mathbb{R}^n$ be an open set. Let $Z$ be the zero set of $f$ in $\Omega$, that is, $Z=\{x\in \Omega\,\vert\, f(x)=0\}$. Assume $Z\neq\emptyset$, $Z\neq\Omega$. Then 
\begin{equation}\label{hyp}
\int_{\Omega\setminus Z}\bigg{|}\frac{\nabla f(x)}{f(x)}\bigg{|}^n\, dx=\infty.
\end{equation}
\end{theorem}
\begin{proof} First we make a subtle observation that in order to prove the theorem it suffices to prove it for $g=f^2$ since 
$$\int_{\Omega\setminus Z}\bigg{|}\frac{\nabla g(x)}{g(x)}\bigg{|}^n dx=2^n\int_{\Omega\setminus Z}\bigg{|}\frac{\nabla f(x)}{f(x)}\bigg{|}^n\, dx.$$
Therefore if $(1.1)$ is true for $g=f^2$, it also true for $f$. Hence we only need to prove $(1.1)$ for functions that are the sqaure of a Lipschitz function. Then by Lemma 2.1, the gradient is 0 whenever the function is 0. For the rest of the proof we assume $\nabla f(x)=0$ whenever $f=0$.
Let 
\begin{equation}
V(x)=\begin{cases}
\Big{|}\frac{\nabla f(x)}{f(x)}\Big{|}\quad x\in \Omega\setminus Z\\
0\quad \quad \quad\,\,\,x\in Z.
\end{cases}
\end{equation}
Note that $V$ is a measurable function in $\Omega$. Assume now by contradiction that \eqref{hyp} is false. Then \begin{equation}
\int_{\Omega} V^n \, dx= \int_{\Omega\setminus Z}\bigg{|}\frac{\nabla f(x)}{f(x)}\bigg{|}^n\, dx< \infty,
\end{equation}
and therefore $V\in L^n(\Omega)$. Choose a point $x_0\in\partial Z$ and a ball $B(x_0,r_0)$ of radius $r_0>0$ centered at $x_0$ such that $B(x_0,r_0)\subset\Omega$. We assume, after a translation and scaling,  that $x_0$ is the origin and the radius $r_0$ is equal to $1$. Hence
\begin{equation}\label{fin}
\int_{B(0,1)}V^n\, dx=\int_{S^{n-1}}\int_{0}^1 V^d(r\omega)r^{n-1}\, drd\omega.
\end{equation}

Since the integral \eqref{fin} is finite, then Fubini's theorem implies that for almost all $\omega\in S^{n-1}$ we have 
\begin{equation}\label{a}
\int_0^1 V^d(r\omega)r^{n-1}\, dr<\infty.
\end{equation}
By Rademacher's theorem, we set $A$ to be the set where $\nabla f(x)$ doest not exist at $x$ so that $A$ is of measure zero. 
Choose $\omega_0\in S^{n-1}$ is such that \eqref{a} holds, that is, $V(r\omega)r^{\frac{n-1}{n}}\in L^d(0,1).$ and at the same time, by Lemma 2.3 we can choose the same $\omega_0\in S^{n-1}$ such that $\nabla f(x)$ exists a.e. on the line (ray) $\{r\omega_0\}$.
Let \[\varphi(t):=f(t\omega_0), \quad t\in [0,1].\]
It is evident that $\varphi(t)$ is Lipschitz in $t$ because of $f$ locally Lipschitz.
Then 
for differential points of $f$, and applying the chain rule there we have 
$$\varphi'(t)=\nabla f(t\omega_0)\cdot\omega_0$$
which implies, for a.e. $t$, 
\begin{equation}
|\varphi'(t)|\leq |\nabla f(t\omega_0)|.
\end{equation}
By the definition of $V$, we have 
\begin{equation}
|\varphi'(t)|\leq V(t\omega_0)|f(t\omega_0)|=V(t\omega_0)|\varphi(t)|\quad \text{for }\,f(t\omega_0)\neq 0.
\end{equation}
However, when $f(t\omega_0)= 0$, we have by the observation at the begining of the proof, $\nabla f(r\omega_0)=0$ and therefore $\varphi'(t)=0$. Hence 
we have shown that 
$$|\varphi'(t)|\leq V(t\omega_0)|\varphi(t)|=V(t\omega_0)t^{\frac{n-1}{n}}|\varphi(t)|t^{-\frac{n-1}{n}}$$ holds for a.e. $t$. By Lemma 2.1, with $\lambda(t)=V(t\omega_0)t^{\frac{n-1}{n}}$ and $p=n$, $\varphi(t)\equiv 0$, and it implies $f\equiv 0$ since $\omega_0$ is arbitrary off a measure zero and $x_0$ in the boundary of $Z$, a contradiction.

\end{proof}

\section{Proof of Theorem 1.1}
Here before we prove Theorem 1.1, which is a simple consequence of Theorem 1.2, we recall that a function belongs to $W^{1,p}(\Omega)$ if $f\in L^p(\Omega)$ and its weak derivative $\nabla f$ belongs to $L^p(\Omega)$. Now we prove Theorem 1.1.
If $\log |f(x)-f(a)|$ does not belong to $L^p(\Omega)$, we are done. If $\log |f(x)-f(a)|$  belongs to $L^n(\Omega)$,
then apply Theorem 1.2 to conclude $\nabla \log |f(x)-f(a)|$ does not belong to $L^n(\Omega)$.

For the second part, first we extend $f$ to be a Lipschitz function defined in $\mathbb{R}^n$. If $f(x)$ vanishes at an point in the closure of $\Omega$, we apply Theorem 1.2 to the exntended function to get a contradiction.

\section{Generalization to Lipschitz mappings}
All results proved so far can be generalized to Lipschitz mappings. First we say $f:\Omega:\to \mathbb{R}^m$ is a Lipschitz mapping if each of the components is a Lipchitz function. Now we can state the mapping version of Theorem 1.2.
\begin{theorem}\label{t1}
Let $\Omega$ be an open set in $\mathbb{R}^n$, and let $f:\Omega:\to \mathbb{R}^m$ be a non-constant locally Lipschitz  mapping on $\Omega$. If the common zero set $\{x\in\Omega: f(x)=0\}$ of the mapping $f$ is not empty, then 
\begin{equation}\label{po}
\int_{\Omega\setminus f^{-1}(0)}\bigg|\frac{\nabla f(x)}{f(x)}\bigg|^{n}\, dx=\infty,
\end{equation}
where we have $|\nabla f|^2=|\nabla f_1|^2+...+|\nabla f_m|^2$ if $f=(f_1, ..., f_m)$.
\end{theorem}

\end{document}